\documentclass[12pt]{amsart}

\usepackage{amssymb,amsmath,latexsym}

\usepackage{amsmath, amsthm, amscd, amsfonts, amssymb, graphicx, color}
\usepackage{marvosym}
\usepackage{paralist}
\usepackage{color}

\usepackage[varg]{pxfonts}

\usepackage{stmaryrd}

\usepackage{paralist}

\newtheorem{lem}{\textbf{Lemma}}[section]

\newtheorem{theorem}[lem]{\textbf{Theorem}}

\theoremstyle{definition}

\theoremstyle{definition}
\theoremstyle{remark}

   \thanks{\!\!\!\!\!\!\!\! \!\!${^{*}}$Corresponding author E-mail address:  sori.e@lu.ac.ir(E. Soori), donal.oregan@nuigalway.ie(D. O$'$Regan), agarwal@tamuk.edu(R. P. Agarwal).}

\begin{document}
\title[A Simple proof      for   Imnang's algorithms ]{A Simple proof      for   Imnang's algorithms }
\author[ Soori,   O$'$Regan, Agarwal]{ Ebrahim Soori$^{*,1}$, Donal O$'$Regan$^{2}$ and  Ravi P. Agarwal$^{3}$}
\address{$^{1}$ Department  of Mathematics, Lorestan University, P.O. Box 465, Khoramabad, Lorestan, Iran, $^{2}$ School of Mathematical and  Statistical Sciences, University of Galway, Galway, Ireland, $^{3}$ Department of Mathematics
 Texas A$\&$M University-Kingsville 700 University Blvd., MSC 172 Kingsville, Texas, USA.}


\date{}


%

\begin{abstract}
In this paper,      a  simple  proof of the convergence of  the recent iterative algorithm  by relaxed $(u, v)$-cocoercive mappings   due to  S. Imnang [S. Imnang,   Viscosity iterative method for a new general system of variational   inequalities in Banach spaces,  J. Inequal. Appl., 249:18 pp., 2013.] is presented.
\end{abstract}


\maketitle

\section{Introduction and Preliminaries}
 In this paper,    a  simple proof   for the convergence of an   iterative algorithm  is  presented which improves and refines the original  proof.

Suppose that $C$ is a nonempty closed   convex subset of a real normed linear  space $E$ and
$E^{*}$ is its dual space.  Suppose that  $\langle.,.\rangle$   denotes the pairing between $E$ and $E^{*}$. The
normalized duality mapping $J: E \rightarrow E^{*}$
is defined by
\begin{align*}
    J(x)=\{f \in E^{*}: \langle x, f \rangle= \|x\|^{2}=\|f\|^{2} \}
\end{align*}
for each $x \in E$.  Let  $U = \{x \in E : \|x\| = 1\}$.   A Banach space $E$ is called  smooth if for all $x \in U$,
there exists a unique functional $j_{x} \in E^{*}$ such that $\langle x, j_{x}\rangle = \|x\|$ and $\|j_{x}\| = 1$ (see \cite{Ag,Soori2021A Strong,Soori2021Strong convergence}).

 For a map $T$  from $ E $ into itself, we denote by $\rm{Fix}(T ) := \{x \in E : x = Tx\}$, the fixed point set of $T$.

   Recall the following well known  concepts:
\begin{enumerate}
\item Suppose that $C$ is a nonempty closed convex subset of a real Banach space $ E $.  A mapping $B: C \rightarrow E$ is  called    relaxed $(u, v)$-cocoercive~\cite{[5]Imnang2013viscosityInterationMethod}, if there exist two constants $u, v > 0$ such that
 \begin{equation*}
    \langle Bx - By , j(x - y)\rangle \geq (-u)\|Bx - By\|^{2}+v\|x - y\|^{2},
 \end{equation*}
   for  all  $x, y \in C$ and $j(x - y) \in J(x - y)$.
\item Suppose that $C$ is a nonempty closed convex subset of a real Banach space $ E $ and $B$ is a  self  mapping on $C$. If there exists      a positive  integer $\alpha$ such that
 \begin{equation*}\label{golf2}
      \| Bx -By\| \geq \alpha\| x -  y\|.
  \end{equation*}
  for  all  $x, y \in C$, then $B$ is called  $\alpha$-expansive.
\end{enumerate}

\begin{lem}\cite{[5]Imnang2013viscosityInterationMethod}\label{ert}
Let $C$ be a nonempty closed convex subset of a real  $2$-uniformly smooth
Banach space $X$ with the $2$-uniformly smooth constant $K$. Let $Q_{C}$ be the sunny nonexpansive retraction from $X$ onto $C$ and let $A_{i} : C \rightarrow X$ be a relaxed $(c_{i}, d_{i})$-cocoercive and
Li-Lipschitzian mapping for $i = 1, 2, 3$. Let $G : C \rightarrow C$ be a mapping defined by
\begin{align*}
  G(x) = & Q_{C}   \big[Q_{C}\big (Q_{C}(x-\lambda_{3}A_{3}x)-\lambda_{2}A_{2}Q_{C}(x-\lambda_{3}A_{3}x)\big)   \\
    & -\lambda_{1}A_{1}Q_{C}\big(Q_{C}(I-\lambda_{3}A_{3})x-\lambda_{2}A_{2}Q_{C}(I-\lambda_{3}A_{3})x\big)\big ]
\end{align*}
 If $\lambda _{i} \leq \frac{d_{i}-c_{i} L_{i}^{2}}{K^{2}L_{i}^{2}}$ for all $i = 1, 2, 3$,  then $G : C \rightarrow C$ is nonexpansive.
\end{lem}
\begin{lem}\cite[Lemma 2.8]{[2]Cai2011Strong convergence} \label{jdjk}
 Suppose that $C$ is a nonempty closed convex subset of a real  Banach space $X$ which is $2$-uniformly smooth, and  the mapping $A:C\rightarrow X$ is relaxed $(c,d)$-cocoercive and $L_{A}$-Lipschitzian. Then  $$\|(I-\lambda A)x-(I-\lambda A)y\|^{2} \leq \|x-y\|^{2}+2(\lambda cL^{2}_{A}-\lambda d+K^{2}\lambda^{2}L^{2}_{A})\|x-y\|^{2},$$
 where $\lambda >0$. In particular, when  $d > c L^{2}_{A}$ and $ \lambda \leq \frac{d-c L^{2}_{A}}{K^{2}L^{2}_{A}} $,  note $I-\lambda A$ is nonexpansive.
\end{lem}

In this paper, using   relaxed $(u, v)$-cocoercive mappings,    a new  proof for   the iterative algorithm \cite{[5]Imnang2013viscosityInterationMethod} is presented.

\section{A simple proof for the theorem }
S. Imnang \cite{[5]Imnang2013viscosityInterationMethod}
considered an iterative algorithm  for finding a common element of the set of
fixed points of nonexpansive mappings and the set of  solutions of a variational
inequality. Our argument will rely on the following lemma.
\begin{lem}\label{sgjm}
     Suppose that $C$ is a nonempty closed convex subset of a Banach space $E$. Suppose that $A: C \rightarrow E$ is    a relaxed $(m, v)$-cocoercive mapping  and $\epsilon$-Lipschitz continuous with $v-m \epsilon^{2}>0$.  Then  $A$  is    a   $(v-m \epsilon^{2})$-expansive mapping.
\end{lem}
\begin{proof}
Since $A$ is $(m, v)$-cocoercive and  $\epsilon$-Lipschitz continuous, for each $x,y \in C$ and $j(x - y) \in J(x - y)$, we have that
\begin{align*}\label{jhgjgl}
  \langle Ax-Ay, j(x-y) \rangle& \geq(-m)\| Ax-Ay\|^{2}+v\|x-y \|^{2} \nonumber \\ \nonumber &\geq (-m \epsilon^{2})\|x-y \|^{2} +v\|x-y \|^{2}\\  &= (v-m \epsilon^{2})\|x-y \|^{2}\geq 0,
\end{align*}
 and   hence
 \begin{equation*}
   \| Ax-Ay\|\geq  (v-m \epsilon^{2})\|x-y \|,
 \end{equation*}
so $A$ is   $(v-m \epsilon^{2})$-expansive.
\end{proof}
The following   theorem   is due to  S. Imnang \cite{[5]Imnang2013viscosityInterationMethod} that  solves the viscosity iterative problem for a new general
system of variational inequalities in Banach spaces:

\begin{theorem}\label{mkkjjhgyj}{ (\rm {\em i.e.}, Theorem 3.1, from~\cite[\S 3, p.7]{[5]Imnang2013viscosityInterationMethod})}
Suppose that $X$ is a Banach space which is     uniformly convex and $2$-uniformly smooth  with
the $2$-uniformly smooth constant $K$,   $C$ is a nonempty closed convex subset of $X$ and
$Q_{C}$ is a sunny nonexpansive retraction from $X$ onto $C$. Assume that  $A_{i}: C  \rightarrow X$ is
relaxed $(c_{i}, d_{i})$-cocoercive and $L_{i}$-Lipschitzian with $0 < \lambda_{i} < \frac{d_{i}-c_{i}L_{i}^{2}}{K^{2}L_{i}^{2}}$
  for each $i = 1, 2, 3$. Suppose that  $f$
is a contraction  mapping with the constant $\alpha \in (0, 1)$ and   $S: C\rightarrow C$, a nonexpansive
mapping such that $ \Omega= F(S)\cap F(G)\neq \emptyset$, where $G$  is  defined as in Lemma 1.1.
Suppose that  $x_{1} \in C$ and $\{x_{n}\}$, $\{y_{n}\}$ and $\{z_{n}\}$ are the following sequences:
\begin{equation*}
\begin{cases}
z_{n}=Q_{C}(x_{n}-\lambda_{3}A_{3}x_{n}),\\
y_{n} = Q_{C}(z_{n}-\lambda_{2}A_{2}z_{n}),\\
x_{n+1} =  a_{n} f(x_{n}) + b_{n} x_{n} +  (1-a_{n}- b_n)SQ_{C}(y_{n}-\lambda_{1}A_{1}y_{n}, &\text{\quad}\\
\end{cases}
\end{equation*}
where  $\{a_{n}\}$ and $\{b_{n}\}$ are two sequences in (0, 1) such that
\begin{compactenum}[(C1)]
\item $\lim_{n \rightarrow \infty}  a_{n} = 0$ and
$\sum_{n=1}^{\infty}a_{n} =\infty$;
\item $0< \displaystyle\lim inf_{n\rightarrow \infty} b_{n} \leq \displaystyle\lim sup_{ n\rightarrow \infty} b_{n} < 1.$
\end{compactenum}
Then $\left\{x_n\right\}$   converges strongly to $q\in \Omega$, which solves the following variational inequality:
$$\langle q-f(q), J(q-p)\rangle \leq 0, \forall f \in \Pi_{C} , p \in \Omega.$$
\end{theorem}

    \textbf{A simple Proof:}\label{bhsskkl54ww}
 Let $i=1, 2,3$. Consider Theorem~\ref{mkkjjhgyj}  and the  $L_{i}$-Lipschitz continuous and   relaxed $(c_{i},d_{i} )$-cocoercive mapping   $A_{i}$ in  Theorem ~\ref{mkkjjhgyj}. From
the condition that $0 < \lambda_{i} < \frac{d_{i}-c_{i}L_{i}^{2}}{K^{2}L_{i}^{2}}$,    we have that $0<1+2(\lambda_{i} c_{i} L^{2}_{i}-\lambda _{i} d_{i}+ K^{2}\lambda^{2}_{i}L^{2}_{i})<1$. Note that    from    Lemma \ref{jdjk}, we have that $I-\lambda _{i} A_{i}$ is nonexpansive when $0<1+2(\lambda_{i} c_{i} L^{2}_{i}-\lambda _{i} d_{i}+ K^{2}\lambda^{2}_{i}L^{2}_{i})$.  Then applying the  coefficients $\alpha_{i}=1+2(\lambda_{i} c_{i} L^{2}_{i}-\lambda _{i} d_{i}+ K^{2}\lambda^{2}_{i}L^{2}_{i})$ in     Lemma \ref{jdjk} we have that  $I-\lambda _{i} A_{i}$ is     $\alpha_{i}$-contraction, for each $i=1,2,3$.  Also, note that $Q_{C}$ is nonexpansive and $I-\lambda _{i} A_{i}$ is   $\alpha_{i}$-contraction, for each $i=1,2,3$. Hence, using    the proof of  \cite[Lemma 2.11]{[5]Imnang2013viscosityInterationMethod},   we conclude that
\begin{align*}
  \|G(x) - G(y)\|  =&\big \| Q_{C}   \big[Q_{C}\big (Q_{C}(I-\lambda_{3}A_{3})x-\lambda_{2}A_{2}Q_{C}(I-\lambda_{3}A_{3})x\big)   \\
    & -\lambda_{1}A_{1}Q_{C}\big(Q_{C}(I-\lambda_{3}A_{3})x-\lambda_{2}A_{2}Q_{C}(I-\lambda_{3}A_{3})x\big)\big ]\\
     & - Q_{C}   \big[Q_{C}\big (Q_{C}(I-\lambda_{3}A_{3})y-\lambda_{2}A_{2}Q_{C}(I-\lambda_{3}A_{3})y\big)   \\
    & -\lambda_{1}A_{1}Q_{C}\big(Q_{C}(I-\lambda_{3}A_{3})y-\lambda_{2}A_{2}Q_{C}(I-\lambda_{3}A_{3})y\big)\big ]\big\|
   \\ \leq &\big \|    Q_{C}\big (Q_{C}(I-\lambda_{3}A_{3})x-\lambda_{2}A_{2}Q_{C}(I-\lambda_{3}A_{3})x\big)   \\
    & -\lambda_{1}A_{1}Q_{C}\big(Q_{C}(I-\lambda_{3}A_{3})x-\lambda_{2}A_{2}Q_{C}(I-\lambda_{3}A_{3})x\big)\\
     & -   \big[Q_{C}\big (Q_{C}(I-\lambda_{3}A_{3})y-\lambda_{2}A_{2}Q_{C}(I-\lambda_{3}A_{3})y\big)   \\
    & -\lambda_{1}A_{1}Q_{C}\big(Q_{C}(I-\lambda_{3}A_{3})y-\lambda_{2}A_{2}Q_{C}(I-\lambda_{3}A_{3})y\big)\big ]\big\|\\
     = & \|(I-\lambda _{1} A_{1})Q_{C}(I-\lambda _{2} A_{2})Q_{C}(I-\lambda _{3} A_{3})x   \\
   &- (I-\lambda _{1} A_{1})Q_{C}(I-\lambda _{2} A_{2})Q_{C}(I-\lambda _{3} A_{3})y\| \\
  \leq & \alpha_{1}\alpha_{2}\alpha_{3} \|x-y\|,
\end{align*}
  and since $0< \alpha_{1}\alpha_{2}\alpha_{3} <1$ then $G$ is an $\alpha $-contraction with $\alpha=\alpha_{1}\alpha_{2}\alpha_{3} $,  hence  from Banach's contraction principle $F(G)$ is a singleton set  and hence, $\Omega$ is a singleton set i.e.,  there exists an element $p \in X$    such that  $\Omega=\{p\} $.  Since   $(d_{i}-c_{i}L_{i}^{2})>0$,   from Lemma \ref{sgjm}, $A_{i}$  is  $(d_{i}-c_{i} L_{i}^{2})$-expansive,
i.e,
 \begin{equation}\label{hgjnb}
      \| A_{i}x -A_{i}y\| \geq (d_{i} -c_{i} L_{i}^{2})\|x - y\|,
  \end{equation}
   in Theorem ~\ref{mkkjjhgyj}.
  The authors  in  \cite[ p.11]{[5]Imnang2013viscosityInterationMethod} proved (see  (3.12) in ~\cite[ p.11]{[5]Imnang2013viscosityInterationMethod}) that
  \begin{equation}\label{asercv}
    \displaystyle \lim_{n}\|A_{3}x_{n}-A_{3}p\|=0,
  \end{equation}
 for $x^{*}=p$. Now, put  $x=x_{n}$  and $y=p$ in  \eqref{hgjnb},  and  from \eqref{hgjnb} and \eqref{asercv},  we have
   \begin{equation*}
     \displaystyle \lim_{n}\|x_{n}-p\|=0.
   \end{equation*}
Hence,    $x_{n} \rightarrow p$.  As a result one of the main claims of  Theorem  ~\ref{mkkjjhgyj}  is established (note $ \Omega=\{p\} $).

 Note that the main aim of Theorem 3.1 in ~\cite{[5]Imnang2013viscosityInterationMethod} are $x_{n} \rightarrow p$  and  $$\langle q-f(q), J(q-p)\rangle \leq 0, \forall f \in \Pi_{C} , p \in \Omega.$$
Next, we show that the main  aim of Theorem 3.1 in ~\cite{[5]Imnang2013viscosityInterationMethod} can be concluded  from  the relations   (3.12) in ~\cite[ page  11]{[5]Imnang2013viscosityInterationMethod} and the proof in Theorem ~\ref{mkkjjhgyj} can  be simplified even further using the above. Note, the part of the proof between  the relations   (3.12) in ~\cite[ page  11]{[5]Imnang2013viscosityInterationMethod} to the end of the proof of Theorem 3.1 can be removed from the proof. Indeed,  since  immediately   from
 (3.12)  in  ~\cite{[5]Imnang2013viscosityInterationMethod}, we conclude that $x_{n} \rightarrow p$,  i.e., the first aim of Theorem 3.1 is concluded. The second aim of the theorem i.e.,
 $$\langle q-f(q), J(q-p)\rangle \leq 0, \forall f \in \Pi_{C} , p \in \Omega,$$
is clear, because $p=q$ ($ \Omega=\{p\} $) and $J(0)=\{0\}$.   Consequently,  the relations between (3.12) in ~\cite[ page  11]{[5]Imnang2013viscosityInterationMethod} to the end  of the proof  of Theorem 3.1  in ~\cite[ page  11]{[5]Imnang2013viscosityInterationMethod} can be removed.

\section{\textbf{Discussion}}
In this paper, a simple proof    for  the  convergence of   an algorithm by  relaxed $(u, v)$-cocoercive
mappings  due to S. Imnang is presented.

 \section{ \textbf{Conclusion}}
In this paper, a  refinement of the proof of  the results due to S. Imnang   is given.

\section{\textbf{Abbreviations}}
Not applicable
\section{\textbf{Declarations}}
\subsection{ Availability of data and material}
Please contact the authors for data requests.
\subsection{Competing interests}
The authors  declare  that they  have  no competing interests.
\subsection{ Funding}
Not applicable
\subsection{Authors’ contributions}
All authors contributed equally to the manuscript, read and approved
the final manuscript.
\subsection{ Acknowledgements}
The first   author is grateful to the University of Lorestan for their
support.

\section{\textbf{Author details}}
$^{1}$ Department  of Mathematics, Lorestan University, P.O. Box 465, Khoramabad, Lorestan, Iran, $^{2}$ School of Mathematical and  Statistical Sciences, University of Galway, Galway, Ireland, $^{3}$ Department of Mathematics
 Texas A$\&$M University-Kingsville 700 University Blvd., MSC 172 Kingsville, Texas, USA.


\end{document}